\documentclass{article}

\usepackage{amssymb, amsmath,amsthm}
\usepackage{mathrsfs}
\usepackage{amscd}
\usepackage[active]{srcltx}
\usepackage{verbatim}
\usepackage{graphicx}
\usepackage{float}

\usepackage[backend=bibtex,style=alphabetic]{biblatex}
\usepackage[utf8]{inputenc}
\usepackage[T1]{fontenc}
\usepackage{lmodern}

\usepackage{hyperref} 
\usepackage{todonotes}
\usepackage{enumerate}
\usepackage{mathtools}
\usepackage{bbm}

\mathtoolsset{showonlyrefs,showmanualtags}

\newcommand\ud{\,\mathrm{d}}

\newcommand{\Nn}{\mathcal{N}}

\newcommand{\Ss}{\mathcal{S}}

\newcommand{\EE}{\mathbb{E}}

\newcommand{\NN}{\mathbb{N}}

\newcommand{\PP}{\mathbb{P}}

\newcommand{\RR}{\mathbb{R}}

\newcommand{\ZZ}{\mathbb{Z}}

\newcommand{\inp}[2]{\langle #1,#2 \rangle}

\newcommand{\Var}{\mathrm{Var}}
\newcommand{\Cov}{\mathrm{Cov}}
\newcommand{\Bin}{\mathrm{Bin}}
\newcommand{\Mult}{\mathrm{Mult}}
\newcommand{\Vol}{\mathrm{Vol}}

\renewcommand{\epsilon}{\varepsilon}

\theoremstyle{plain}
\newtheorem{theorem}{Theorem}[section]
\theoremstyle{remark}
\newtheorem{remark}[theorem]{Remark}

\theoremstyle{plain}

\newtheorem{lemma}[theorem]{Lemma}
\newtheorem{proposition}[theorem]{Proposition}

\newtheorem{conjecture}[theorem]{Conjecture}
\newtheorem{assumption}[theorem]{Assumption}

\numberwithin{equation}{section}

\allowdisplaybreaks

\begin{document}

\title{Giant component in the configuration model under geometric constraints}


\author{\renewcommand{\thefootnote}{\arabic{footnote}}
Ivan Kryven\footnotemark[1]~
and
Rik Versendaal\footnotemark[1]
}

\footnotetext[1]{
Mathematical Institute, Utrecht University, P.O. Box 80010, 3508 TA Utrecht, The Netherlands, E-mail: \texttt{i.v.kryven/r.versendaal@uu.nl}.
}

\date{\today}

\maketitle

\begin{abstract}
    We study the emergence of a giant component in the configuration model subject to additional constraints on the edges. We partition a $d$-dimensional torus into a cubic lattice with a diverging number of compartments containing vertices and  allow only local edges inside and between neighbouring compartments.  We show that, when the number of vertices per compartment grows quickly enough, a giant component emerges under similar conditions as for the standard configuration model. Conversely, when the compartment sizes are fixed, our model might not feature a giant component even if the standard configuration model does have one.
 Locally, our model resembles the configuration model, while globally, it has properties more akin to a $d$-dimensional lattice. Nonetheless the model remains analytically tractable using multitype branching processes with infinite number of types and opens new potential ways to study percolation in graphs with geometric properties.\\
    
    \textit{Keywords:} Configuration model, giant component, multitype branching process, concentration inequalities, geometric networks\\
    
    MSC2020 Classes: 05C80, 60J80
\end{abstract}

\tableofcontents

\section{Introduction}
Since the classical random graph model was first introduced by Erd\H{o}s and R\'enyi, many alternative models were studied by adding constraints to this random graph. In the configuration model one can impose an arbitrary  degree sequence. Such a choice may  affect the global connectivity of the random graph, inducing a so-called phase transition \cite{MR98,JL09,BR15}. That is, the model may or may not feature a giant connected component that involves a positive fraction of vertices depending on the chosen degree sequence. In random geometric graphs, the vertices have coordinates defined by a point process and are connected based on their proximity.  These graphs also feature a similar phase transition \cite{Pen03}, which seems to be a property of the embedding metric space. At the same time, the embedding space also induces a certain degree distribution, which one cannot control independently. In general, even though both models feature phase transition-like behaviour, there are only a few results allowing to study random geometric graphs that have a given degree distribution. 
One approach was suggested in the small world graphs \cite{watts1998collective,barbour2001small}, where a regular circular lattice or a continuous circle is randomly rewired by adding shortcuts to obtain an object that retains some of the original geometric properties while having a controlled degree distribution.
Random graph models in which both the degree distribution and geometrical features can be controlled are relevant when modelling real networks having some spatial content. \\

Our aim is to provide a simple geometric generalisation of the configuration model by additionally forbidding some pairs of vertices to be connected, hence inducing a notion of a metric. We study the following model: We consider $k\in\mathbb N$ compartments arranged into a $d$-dimensional cubic lattice on a torus and distribute the vertices equally over these compartments. Every compartment has $2d$ neighbouring compartments. We then only allow an edge to connect pairs of vertices belonging to the same or neighbouring compartments. This makes it more difficult for a giant component to emerge, as connections can only be made locally on the $d$-torus.  Our model is furthermore motivated by studying networks with geometric constraints. Since we are only allowed to connect vertices from neighbouring compartments,  such construction may be viewed as a random geometric graph on ${\bf Z}_k^d$ that has  a given degree distribution. When ${\bf Z}_k^d$ is embedded in the $d$-torus, the larger $k$ is, the closer the connected vertices are together.  However useful ${\bf Z}_k^d$ model is for applications, we also hope that the techniques used in this study will in future inspire investigation of the classical random geometric graphs, for example in $\mathbb R^d$.\\

Our technique relies on the idea that the exploration of components in the random graph can be linked to a branching process. However, in comparison to the standard setting, this connection is only valid for a small number of exploration steps, and consequently, this only allows us to prove that locally-large components emerge.  Since the number of compartments tends to infinity, these components become `more and more local'. 
A delicate step in this reasoning is to show that a growing number of local components will be simultaneously present with high probability, so that one giant `super' component can be formed from infinitely many local components connecting together.
 To do this, we introduce a countably infinite number of types into the exploration process to track how each explored component spreads through the different compartments.  This allows us to connect the exploration process to a multitype branching process, where the type of a vertex represents the compartment it belongs to.  This connection allows us to prove that locally-large components spread through a sufficient number of compartments, occupying a positive fraction of the vertices in each of them, and that these components are connected to each other with high probability -- resulting in a giant component.
 To this end, we need to analyse the probability that a local connected component emerges much more carefully in comparison to the standard setting and to obtain precise quantitative bounds for these probabilities.
 \\ 

We believe that our multitype exploration technique can be reapplied to many different  settings beyond the ${\bf Z}_k^d$ arrangement of compartments. For example, one may use heterogeneous compartments to impose clustered structure in a network or long range dependencies between vertex degrees.\\

This article is structured as follows. In Section \ref{section:model} we introduce the model we are studying and state our main theorem. As with the standard configuration model, the proof of our main theorem relies on building a connection with an exploration process, which in our case is a multitype branching process. We introduce this process in Section \ref{section:branching}, where we also derive some relevant properties. With all preparations done, Section \ref{section:proof} is dedicated to proving our main theorem, which is done in a number of propositions. Finally, in Section \ref{section:counterexample} we provide an example that shows that the geometric constraints give rise to different behaviour compared to the configuration model without geometric constraints.


\section{Compartment model on a $d$-dimensional torus}\label{section:model}

\begin{figure}
\begin{center}
\includegraphics[width=0.65\textwidth]{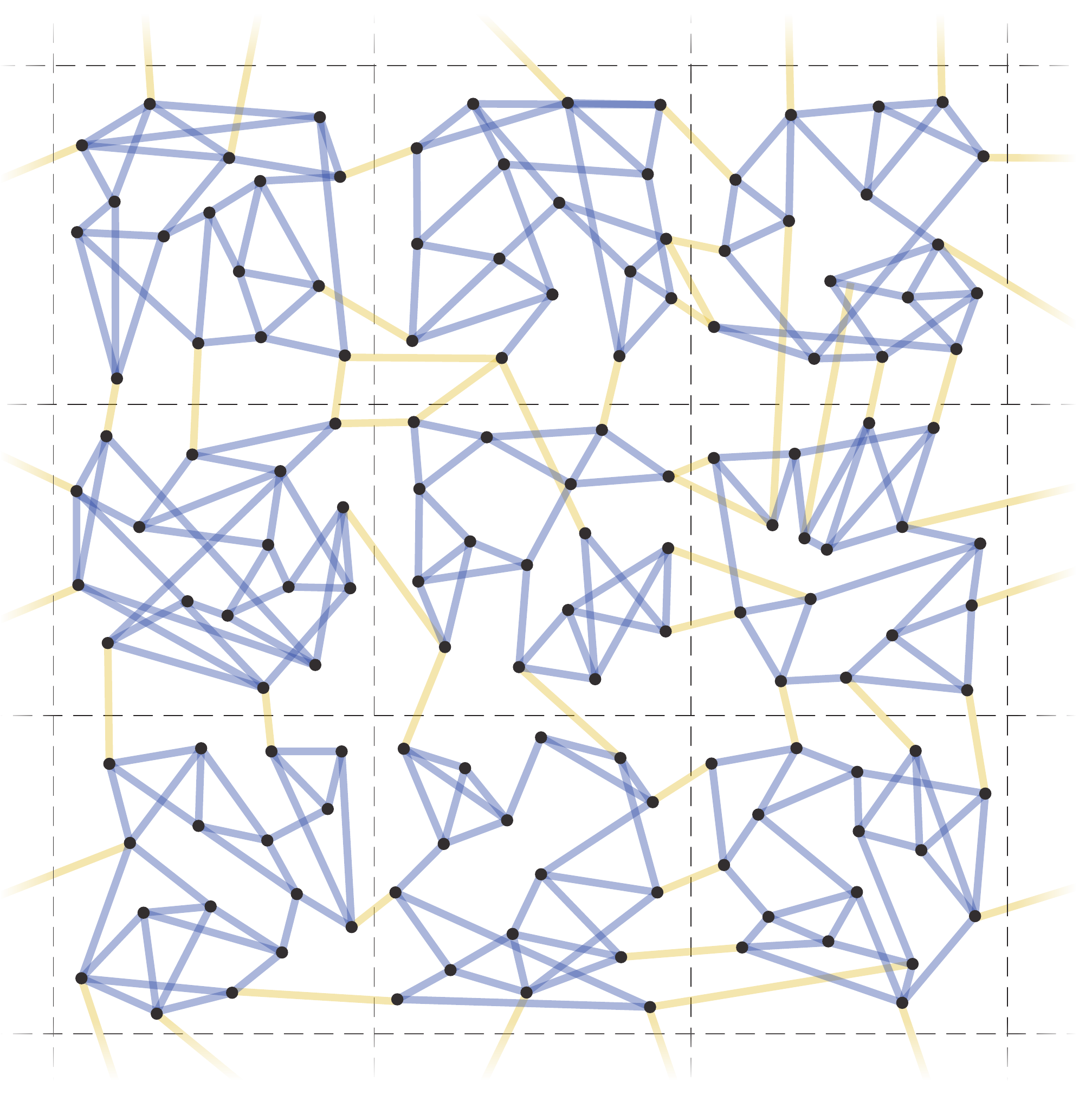}
\caption{Fragment of the configuration model with compartments arranged into a square lattice.}
\label{fig}
\end{center}
\end{figure}
The model we are studying is illustrated in Figure \ref{fig} for $d=2$.
In general, let $d \geq 1$ be an integer. For every $n \in \NN$, we consider $k(n)^d$ compartments $\{C_i^n\}_{i \in [k(n)]^d}$, where $[k(n)] := \{1,\ldots,k(n)\}$. We assume each compartment contains $m(n)$ vertices and we define $V_n = \bigcup_{i \in [k(n)]^d} C_i$ as the set of vertices. Our aim is to study graphs on $V_n$ satisfying two types of constraints:

\begin{enumerate}
    \item \emph{Constraint on allowed connections:} Vertices $x,y \in V_n$ can only be connected if $x \in C_i, y \in C_j$ with $|i-j|_1 \leq 1$. Here, we identify $k(n)$ with 0, allowing for instance also edges between compartments $C_{(i,k(n))}$ and $C_{(i,1)}$ with $i \in [k(n)-1]^d$. This results in the $d$-dimensional torus structure. 
    \item \emph{Degree constraint}: The vertices have prescribed degrees, given by a sequence $d_n = \{d(1),\ldots,d(k(n)m(n))\}$ of non-negative integers. We will refer to $d_n$ as the \emph{degree sequence}.
\end{enumerate}

In principle, the random graph $G_n = (V_n,E_n)$ satisfying the above-mentioned constraints can be constructed algorithmically as follows: The degree of a vertex $x \in V_n$ is represented by $d_n(x)$ half-edges. At each iteration we choose uniformly a pair of half-edges which are allowed to be connected together. We repeat this until no matches can be made. We refer to this model as the \emph{compartment model on $d$-torus}. 

First of all, note that $G_n$ is in general a multi-graph since we do not exclude self-loops or multi-edges. Furthermore, it might happen that we do not satisfy the full degree sequence, even if we assume the sum of the degrees is even. However, when the construction terminates, at most one half-edge per compartment will be unmatched. This will be no problem, since we will be assuming that $m(n)$, the amount of vertices per compartment, tends to infinity.\\  

The compartment model on the $d$-torus is closely related to a random geometric graph on the torus. Indeed, the number of compartments  $k(n)$ is related to the distance between vertices that can be connected. However, in the compartment model, the neighbourhoods of the vertices are homogenized, in the sense that each vertex in a compartment has the same neighbours it can be connected to.

\subsection{Main theorem}

Our main result is concerned with providing sufficient conditions under which the random graphs $G_n$ asymptotically contain a giant component with high probability. Moreover, we will also determine its size. In this section we collect all of our assumptions.\\

First of all, we assume that $V_n$ asymptotically contains $n$ vertices. Noting that $|V_n| = k(n)^dm(n)$ we therefore assume that
$$
\lim_{n\to\infty} \frac{k(n)^dm(n)}{n} = 1.
$$

Second, we will assume that $\lim_{n\to\infty} k(n) = \infty$. On the one hand, this reflects the idea that vertices are only allowed to be connected when they are very close together. On the other hand, this assures that our model is clearly distinguished from the standard configuration model. Indeed, if we only have finitely many compartments, then it should be possible to deduce from the standard configuration model that a giant component emerges locally. It then remains to show that (finitely many) of those connect together with high probability.\\



In Section \ref{section:counterexample} we will see that if the number of vertices per compartment becomes fixed, then a giant component does not necessarily emerge, 
even if the degree sequence satisfies the conditions of our main theorem. Therefore, we will assume that $m(n)$, the number of vertices per compartment, tends to infinity. In particular, we will assume that
$$
\lim_{n\to\infty} \frac{n}{m(n)^k} = 0
$$
for some $k \in \NN$, which can be equivalently stated as there exist $C,\alpha > 0$ such that
$$
m(n) \geq Cn^\alpha.
$$

Apart from assumptions on the graph structure, we also need assumptions on the degree sequences $d_n$. These are the same for the standard configuration model, see e.g. \cite{BR15,Hof17,Dur07}. In what follows, we denote by $n_j(d_n,C_i^n)$ the amount of vertices of degree $j$ in compartment $C_i^n$. Furthermore, we define $\mu_n(d_n,C_i^n)$ by 
$$
\mu_n(d_n,C_i^n) := \frac12\sum_{x \in C_i^n} d(x) = \frac12\sum_{j=1}^\infty jn_j(d_n,C_i^n).
$$
Using this notation, we make the following assumption on the convergence of the degree sequence $d_n$. 

\begin{assumption}[Convergent degree sequence]\label{assumption:convergent_deg_seq}
The degree sequence $d_n$ converges to a distribution $D$ in the following sense:
\begin{enumerate}
    \item For every $\epsilon > 0$ there exists an $N$ such that for all $n \geq N$ and all $i \in [k(n)]^d$ we have
    $$
    \left|\frac{n_j(d_n,C_i^n)}{m(n)} - \PP(D = j)\right| < \epsilon
    $$
    for all $j$.
    \item For every $\epsilon > 0$ there exists an $N$ such that for all $n \geq N$ and all $i \in [k(n)]^d$ we have
    $$
    \left|\frac{\mu_n(d_n,C_i^n)}{m(n)} - \frac{\EE(D)}{2}\right| < \epsilon.
    $$
\end{enumerate}
\end{assumption}

We are now ready to state the main theorem.

\begin{theorem}\label{theorem:giant_component}
Consider the compartment model on the $d$-torus with $k(n)^d$ compartments with $m(n)$ vertices each, so that $\lim_{n\to\infty} \frac{m(n)k(n)^d}{n} = 1$. Assume that $\lim_{n\to\infty} k(n) = \infty$ and that there exists $k \in \NN$ such that
$$
\lim_{n\to\infty} \frac{n}{m(n)^k} = 0.
$$
Furthermore, for every $n$ let $d_n$ be a degree sequence on $m(n)k(n)^d$ vertices satisfying Assumption \ref{assumption:convergent_deg_seq} with distribution $D$. Assume $E(D(D-2)) > 0$ and that there exists a $t > 0$ such that $\EE(e^{tD}) < \infty$. If we denote by $L_1(G_n)$ the largest component in $G_n$, then there exists a $\rho \in [0,1)$ such that
$$
\lim_{n\to\infty} \frac{L_1(G_n)}{n} = 1 - \rho
$$
in probability. Furthermore, with high probability, there is no other cluster of size more than $\beta\log m(n)$ for some $\beta > 0$.
\end{theorem}

\begin{remark}\label{remark:size_biased}
The constant $\rho$ in Theorem \ref{theorem:giant_component} can be determined from the distribution $D$. More precisely, we define the distribution $D^*$ by $\PP(D^* = i) = \frac{i\PP(D = i)}{\EE(D)}$, the so called \emph{size-biased degree distribution}. We can then interpret $\rho$ as the extinction probability of the Galton-Watson tree where the root has offspring distribution $D$, and all other individuals have offspring distribution $Z_D = D^* - 1$. The condition $E(D(D-2)) > 0$ implies that $\EE(Z_D) > 1$. In particular, this implies that the Galton-Watson tree survives with positive probability, implying that $\rho < 1$. 
\end{remark}



\section{Branching processes} \label{section:branching}

Studying components in random graphs is closely related to studying branching processes. This occurs when we explore components of a graph from a given vertex. Then, the next generation of the branching process  resembles the neighbours in the graph of the current generation. Such exploration may traverse from one compartment to the other. Therefore, we will make use of a multitype branching process to keep track of the compartment we are in. In this section we will shortly introduce these processes, and collect some necessary results. For a more thorough treatment, see e.g. \cite{AN72,AL06}.  

\subsection{Galton-Watson tree}

The prototypical example of a branching process is the Galton-Watson tree, which models the evolution of a population in which every individual of a generation gets a random number of children. Furthermore, it is assumed that the number of children of different individuals are independent, and follow the same distribution.\\

More precisely, let $D$ be a probability distribution on the nonnegative integers and denote by $Z_n$ the number of individuals in generation $n$. For every $n$, let $X_1^n,\ldots,X_{Z_n}^n$ be independent random variables with distribution $D$. Then
$$
Z_{n+1} = \sum_{i=1}^{Z_n} X_i^n. 
$$

An important question regarding such processes is whether they become extinct or grow on indefinitely. We define the \emph{extinction probability} by
$$
\rho(D) = \lim_{n\to\infty} \PP(Z_n = 0).
$$

If $\EE(D) < 1$ then the process becomes almost surely extinct, i.e. $\rho(D) = 1$. If $\EE(D) > 1$ then the process has a positive probability to grow on indefinitely. Moreover, this probability can be computed from the generating function of the distribution $D$. In particular, the extinction probability is the largest solution in $[0,1]$ of the equation
$$
x = \sum_{i=0}^\infty \PP(D = i)x^i.
$$

\subsection{Multitype branching processes} \label{subsection:multitype_branching}

For our purposes, it is not sufficient to understand how large components grow. We also need information on how components spread through different compartments. In order to study this, we consider a branching process with types, where type $i \in I^d \subset \ZZ^d$  of a vertex represents its compartment. We denote generation $n$ of the branching process by a matrix $Z_n$ of size $|I|^d$, where $Z_n(i)$ is the number of individuals of type $i$ in generation $n$. We denote by $|Z_n|$ the size of generation $n$, i.e.,
$$
|Z_n| = \sum_{i \in I^d} Z_n(i).
$$
For every type $i \in I^d$ we have an offspring distribution $D_i$, which is now a distribution on matrices representing the types of the offspring. For every $n$ and every $i \in I$, let $X_1^{n,i},\dots,X_{Z_n(i)}^{n,i}$ be independent random variables with distribution $D_i$. We then have that
$$
Z_{n+1} = \sum_{i\in I} \sum_{j=1}^{Z_n(i)} X_j^{n,i}.
$$

When $I$ is finite, one looks at the matrix $M$ of expected offspring to study the extinction of such processes. If we, for instance, assume that $M^k$ has only positive entries for some  sufficiently large $k$, then the largest eigenvalue $\rho_{max}$ of $M$ determines whether extinction occurs almost surely or whether there is some positive probability that the tree grows indefinitely,  see e.g. \cite{Har63,Dur07}. When $I$ is countably infinite, the conditions for extinction are  more subtle, and we refer to \cite{Moy64,HLN13} among others. 

\subsubsection{Assigning types independently}

We are specifically interested in the case where each offspring of a vertex is independently assigned a type according to some distribution. In this case, the offspring distribution is a multinomial distribution. Our claim is that the distribution of individuals over the types in generation $n$ of such a multitype branching process can be found by running a number of $n$-step independent random walks equal to the size of the $n$-th generation. 

More precisely, let $I^d \subset \ZZ^d$ be the state space. Let $N$ be a random variable taking values in the nonnegative integers, denoting the number of children an individual will have. Furthermore, for $i \in I^d$, let $p^i = (p_j^i)_{j\in I^d}$ be a probability distribution on $I^d$. Let $D_i$ denote a multinomial distribution with $N$ trials and probability vector $p^i$, which we will take as offspring distribution of a type $i$ individual. Finally, we denote by $(Z_n)_n$ the associated multitype branching process.

Let us now define the inhomogeneous random walk $(S_n)_n$ with which we want to compare the branching process $(Z_n)_n$. Since the walk is inhomogeneous, we will construct it recursively. Let $S_0$ be distributed according to a uniformly random individual of $Z_0$. Now, if $S_n$ is given, we define $S_{n+1}$ as the random variable with distribution $p^{S_n}$. The following proposition relates this random walk to the branching process with multinomial offspring distribution.

\begin{proposition}\label{prop:Branching_process_RW}
Let $(Z_n)_n$ be a multitype branching process with multinomial offspring distribution. Let $(S_n)_n$ be the associated random walk defined above, and let $S^1,S^2,\ldots,S^{|Z_n|}$ be independent copies of $S_n$. For $i \in I^d$, let $E_i$ be the matrix such that $E_i(j) = \delta_{ij}$. Then $Z_n$ is in distribution equal to 
$$
\Ss_n = \sum_{i=1}^{|Z_n|} E_{S^i}.
$$
\end{proposition}
\begin{proof}
We will prove this using induction on $n$. First of all, note that $\Ss_0$ is equal in distribution to $Z_0$, since $S_0$ is distributed according to a uniformly random individual of $Z_0$.

Now suppose that $\Ss_n$ has the same distribution as $Z_n$. Observe that by definition, the random variables $Z_n(i)$ for $i \in I^d$ are independent. Therefore, if $X_1,\ldots,X_{|Z_{n+1}|}$ are independent samples taken uniformly from the population $Z_{n+1}$, then
$$
Z_{n+1} = \sum_{j=1}^{|Z_{n+1}|} E_{X_j}
$$
in distribution. Now note that the distribution of $X_j$ is equal to $p^{Y_n}$ where $Y_n$ is a uniform sample from the population $Z_n$. Since $Z_n$ is equal in distribution to $\Ss_n$, this means that $X_j$ has distribution $p^{S_n}$. As a consequence, we find that $X_j$ is equal in distribution to $S_{n+1}$. Putting everything together, we conclude that $Z_{n+1} = \Ss_{n+1}$ in distribution.
\end{proof}




The above identification of the multitype branching process as a sum of random walks is useful in deriving properties of the distribution of its $n$-th generation. In particular, we consider the specific case where $I = \ZZ$ and 
$$
p^i = \frac{1}{2d+1}\sum_{|j-i|_1 \leq 1} E_j.
$$
One can show that in generation $n$, all types that are at most at distance $\sqrt n$ from the starting type $Z_0 = E_0$ are present with a significant fraction. Before we can turn this in a rigorous statement, we first need the following result on the associated random walk.

\begin{lemma}\label{lemma:spreading_prob}
Let $X_1,X_2,\ldots$ be a sequence of i.i.d. random variables with $\PP(X_i = \pm e_j) = \PP(X_i = 0) = \frac1{2d+1}$ for all $j = 1,\ldots,d$, where $\{e_1,\ldots,e_d\}$ denotes the standard basis of $\RR^d$. Define $S_n = \sum_{i=1}^n X_i$. Then there exists a $\delta > 0$ such that for $n$ large enough we have
$$
\PP(S_{n^2} = v) \geq \delta\PP(S_{n^2} = 0).
$$
for all $v \in \ZZ^d$ with $|v|_1 = n$. Moreover, for $n$ large enough we have
$$
\PP(S_n = v) \geq \delta \PP(S_n = 0).
$$
for all $v \in \ZZ^d$ with $|v|_1 \leq \sqrt n$.
\end{lemma}
\begin{proof}
By the Kolmogorov-Rogozin inequality (\cite[Theorem 3]{Ess66}, see also \cite{Kol58,Rog61}) there exists a constant $C > 0$ such that
$$
\PP(S_{n^2} = 0) \leq Cn^{-d}.
$$

We are done once we show that for $v \in \ZZ^d$ with $|v|_1 = n$ we have 
$$
\PP(S_{n^2} = v) \geq cn^{-d}
$$
for some $c > 0$ (independent of $v$). To this end, note that $\EE(X_1) = 0$ with covariance matrix $\Sigma$ given by 
$$
\Sigma_{ij} = 
\begin{cases}
\frac{2d}{(2d+1)^2}     &   i = j\\
-\frac{1}{(2d+1)^2}     &   i \neq j\\
\end{cases}.
$$

Therefore, by the central limit theorem we find that
$$
\frac{1}{n}S_{n^2} \Rightarrow \Nn\left(0,\Sigma\right).
$$

Now define for $0 \leq r < s$ the annulus $A_1(r,s) \subset \RR^d$ by
$$
A_1(r,s) = \{x \in \RR^d| r \leq |x|_1 \leq s\}.
$$

We then find that asymptotically we have
$$
\PP(n \leq |S_{n^2}|_1 \leq (1+\epsilon) n) \geq (2\pi)^{-\frac d2}(\det\Sigma)^{-\frac 12}\int_{A_1(1,1+\epsilon)} \exp\left(-\frac12\inp{x}{\Sigma^{-1}x}\right)\ud x.
$$

To compute this further, we use the spectrum of $\Sigma$. It turns out that this is given by $\lambda_1 = \frac{d+1}{(2d+1)^2}$ with multiplicity 1 and $\lambda_2 = \frac{1}{2d+1}$ with multiplicity $d-1$. This implies that
$$
\det\Sigma = \frac{(d+1)}{(2d+1)^(d+1)}.
$$
Furthermore, we have that
$$
|\inp{x}{\Sigma^{-1}x}| \leq ||\Sigma^{-1}||_2||x||_2^2 = \lambda_{\max}(\Sigma^{-1})||x||_2^2 = \frac{(2d+1)^2}{d+1}||x||_2^2 \leq \frac{(2d+1)^2}{d+1}||x||_1^2.
$$
Therefore, on $A(1,1+\epsilon)$, we have
$$
|\inp{x}{\Sigma^{-1}x}| \leq \frac{(2d+1)^2}{d+1}(1+\epsilon)^2 = C(d)(1+\epsilon)^2.
$$
Collecting everything, we find that
$$
\PP(n \leq |S_{n^2}|_1 \leq (1+\epsilon) n) \geq (2\pi)^{-\frac d2}(d+1)^{-\frac 12}(2d+1)^{\frac12(d+1)}\Vol(A_1(1,1+\epsilon))e^{-\frac12C(d)(1+\epsilon)^2}.
$$

Now note that
$$
\Vol(1,1+\epsilon) \sim \epsilon^d.
$$
We thus have constants $C_1(d),C_2(d) > 0$ such that
$$
\PP(n \leq |S_{n^2}|_1 \leq (1+\epsilon) n) \geq C_1(d)\epsilon^de^{-C_2(d)(1+\epsilon)^2}
$$

Furthermore, we have
$$
|\{v \in \ZZ^d | n \leq |v|_1 \leq (1+\epsilon)n\}| \sim (\epsilon n)^d.
$$
Now, making use of the fact that $\PP(S_{n^2} = v)$ is decreasing in $|v|_1$ and $\PP(S_{n^2} = v) = \PP(S_{n^2} = w)$ whenever $|v|_1 = |w|_1$, we find from the above that
$$
\PP(S_{n^2} = v) \geq C_1(d)\frac{1}{(\epsilon n)^d}\epsilon^de^{-C_2(d)(1+\epsilon)^2} = C_1(d)n^{-d}e^{-C_2(d)(1+\epsilon)^2}
$$
whenever $|v|_1 = n$. This shows we can take
$$
c = C_1(d)e^{-C_2(d)(1+\epsilon)^2},
$$
which proves the first statement.

The second statement now follows from the observation that $\PP(S_n = v)$ is decreasing when $|v|_1$ is increasing. 
\end{proof}

Using concentration inequalities, we now show that if we start with a type 0 individual, then in generation $n$ all types $i$ at distance at most $\sqrt n$ from the origin form a positive fraction of the total size of generation $n$. We have the following proposition.

\begin{proposition}\label{prop:multitype_concentration}
Let $Z_n$ be a multitype branching process with types $I = \ZZ^d$. Assume the offspring distribution is multinomial with parameters $N$ and $p = (p^i)$, where $p^i = \frac1{2d+1}\sum_{|j-i|_1 \leq 1} E_j$. Suppose $M_n$ is such that $\PP(|Z_n| \geq M_n) > 0$. Then there exists a $\delta > 0$ such that for every integer $0 \leq r \leq \sqrt n$ and for $n$ large enough we have
\begin{multline}
\PP\left(Z_n(k) \geq \delta\frac{M_n}{n^{\frac d2}} \mbox{ for } |k|_1 \leq r\middle| |Z_n| \geq M_n \right)
\geq 1 - 2(2r+1)^d\exp\left(-\frac{\delta^2M_n}{8n^d}\right).
\end{multline}
\end{proposition}
\begin{proof}
Let $X_1,X_2,\ldots,X_{|Z_n|}$ be independent copies of the random walk $S_n$ in Lemma \ref{lemma:spreading_prob}. By Proposition \ref{prop:Branching_process_RW} we have that $Z_n$ is equal in distribution to
$$
\Ss_n = \sum_{i=1}^{|Z_n|} E_{X_i}.
$$

Now define
$$
\overline\Ss_n = \sum_{i=1}^{M_n} E_{X_i}.
$$

Note that $\Ss_n$ is larger in distribution than $\overline\Ss_n$. As a consequence, we find that
\begin{align*}
\MoveEqLeft
\PP\left(Z_n(k) \geq \delta\frac{M_n}{\sqrt n} \mbox{ for } |k|_1 \leq r\middle| |Z_n| \geq M_n \right)
\\
&=
\PP\left(\Ss_n(k) \geq \delta\frac{M_n}{\sqrt n} \mbox{ for } |k|_1 \leq r\middle| |Z_n| \geq M_n \right)
\\
&\geq
\PP\left(\overline\Ss_n(k) \geq \delta\frac{M_n}{\sqrt n} \mbox{ for } |k|_1 \leq r\right).
\end{align*}

Now define the function $F:\RR^{M_n} \to \RR^{\ZZ^d}$ given by
$$
F(x_1,\ldots,x_{M_n}) = \sum_{i=1}^{M_n} E_{x_i}
$$
Writing $I_r = [-r,r] \cap \ZZ$, we define from this the function $\tilde F_r:\RR^{M_n} \to \RR^{I_r^d}$ given by
$$
\tilde F_r(x_1,\ldots,x_{M_n}) = \sum_{|l|_1\leq r}F_l(x_1,\ldots,x_{M_n})E_l. 
$$
Here, $F_l$ denotes the $l$-th coordinate of $F$. 

If we change a variable $x_i$, then at most one entry in the image of $\tilde F_r$ is increased by 1 while at most one other is decreased by 1. Therefore, we can apply Theorem \ref{theorem:McDiarmid_vector} with $c_i = 1$ for all $i \in I_r^d$. Since $|I_r^d| = (2r+1)^d$, this gives us that
$$
\PP\left(||\overline\Ss_n - \EE(\overline\Ss_n)||_{L^\infty(I_r^d)} \geq \epsilon\right) \leq 2(2r+1)^d\exp\left(-\frac{\epsilon^2}{2M_n}\right)
$$
for every $\epsilon > 0$.

Now we can compute $\EE(\overline\Ss_n(k)) = M_n\PP(S_n = k)$. Therefore, by Lemma \ref{lemma:spreading_prob} there exists a $\delta > 0$ such that
$$
\EE(\overline\Ss_n(k)) \geq \delta M_n\PP(S_n = 0).
$$
for $n$ large enough and $|k|_1 \leq \sqrt n$. By then central limit theorem and the fact that $\PP(S_n = k)$ is largest when $k = 0$, we find that
$$
\PP(S_n = 0) \geq cn^{-\frac d2}
$$
for some $c > 0$. Therefore, by shrinking $\delta$ sufficiently, we find that
$$
\EE(\overline\Ss_n(k)) \geq \delta\frac{M_n}{n^{\frac d2}}.
$$

Now 
\begin{align*}
\MoveEqLeft
\PP\left(\overline\Ss_n(k) \geq \frac{\delta}{2}\frac{M_n}{n^{\frac d2}} \mbox{ for } |k|_1 \leq r\right)
\\
&\geq 
\PP\left(\left|\overline\Ss_n(k) - \EE(\overline\Ss_n(k))\right| \leq \frac{\delta}{2}\frac{M_n}{n^{\frac d2}} \mbox{ for } |k|_1 \leq r \right)
\\
&\geq
\PP\left(||\overline\Ss_n - \EE(\overline\Ss_n)||_{L^\infty(I_r^d)} \leq \frac{\delta}{2}\frac{M_n}{n^{\frac d2}}\right)
\\
&=
1 - \PP\left(||\overline\Ss_n - \EE(\overline\Ss_n)||_{L^\infty(I_r^d)} > \frac{\delta}{2}\frac{M_n}{n^{\frac d2}}\right)
\\
&\geq
1 - 2(2r+1)\exp\left(-\frac{\delta^2M_n}{8n^d}\right)
\end{align*}
which concludes the proof.
\end{proof}




\section{Proof of Theorem \ref{theorem:giant_component}} \label{section:proof}

In this section we prove Theorem \ref{theorem:giant_component}. The proof relies on a similar idea as the proof in \cite{Dur07} for the standard configuration model. To study the components of the random graph $G_n$ from the compartment model, we use an exploration process. However, a major complication that arises in our case is that we can only use this to find large components locally, because the compartment structure restricts the neighbours of vertices we are exploring. Since the total number of compartments $k(n)$ diverges, these large components become `more and more local'.

The main idea is now to show that many (i.e. tending to infinity) such locally-large components together form one giant component. However, this means that we have to prove that a diverging number of such local components exist simultaneously with high probability. Therefore, we need to obtain precise quantitative bounds on the probability that the exploration process finds a sufficiently large (local) component. 

To study this exploration process, we connect it to a branching process. In particular, we also need to track how the explored component spreads through the different compartments. For this, our idea is to introduce types, where each type represents a component. We then connect the exploration process through the compartments to a multitype branching process. A similar argument was recently used in \cite{Siv14} to study site-percolation on the $d$-dimensional Hamming torus. There, the types represent the $d$ directions in which exploration can take place. We use the types in a completely different way, with the multitype branching process tracking 'higher-order' structure of the graph, that is its compartmental structure.
This representation allows us to show that all locally-large components connect together with high probability to form one large component that spreads through all compartments of the $d$-dimensional torus. Finally, the proof is concluded by showing that this large component is actually a giant component by determining its size.

\subsection{Exploration process} \label{subsection:exploration}

In order to study the growth of components in the graph $G_n = (V_n,E_n)$ of the compartment model, we will explore them iteratively. To do this, we start at a vertex $v \in V_n$ and reveal its neighbours. After that, we consider each of these newly revealed vertices and reveal their neighbours, and so on. In particular, in the exploration, we keep track of the compartment to which each vertex belongs. Let us now define this process rigorously.\\

Recall that the graph $G_n = (V_n,E_n)$ consists of $k(n)^d$ compartments $\{C_i^n\}_{i \in [k(n)]^d}$, each containing $m(n)$ vertices. Let now $v \in C_j^n$ be some vertex in the graph $G_n$. The \emph{exploration process} started at $v$ is a sequence of tuples $(R_l,A_l,U_l)$ of hypermatrices of dimension $k(n)^d$ constructed recursively. Here $R_l(i)$ denotes the set of explored vertices in compartment $C_i^n$, $A_l(i)$ the set of active vertices in compartment $C_i^n$, i.e., those that we have already revealed, but not yet explored, and $U_l(i)$ are the other, yet unseen vertices in compartment $C_i^n$. We initialize the process by setting $R_0(i) = \emptyset$ for all $i$, $A_0(j) = \{v\}$ and $A_0(i) = \emptyset$ for all $i \neq j$ and $U_0(j) = C_j^n - \{v\}$ and $U_0(i) = C_i^n$ for $i \neq j$. Now, at every iteration, we define $A_{l+1}$ to be all neighbours of vertices in $A_l$ which are in $U_l$. We then set $R_{l+1} = R_l \cup A_l$ and $U_{l+1} = U_l \setminus A_{l+1}$, where the set operations have to be interpreted element-wise.

A somewhat related exploration process  was used in  \cite{Siv14}. However, there the types are assigned while the exploration is running, while in our case, the types are known beforehand. Furthermore, in \cite{Siv14}, one vertex is explored at a time, while we consider an entire generation at once.

Using this exploration process, we want to analyse how large the component we explore grows. In order to do this, we need to find a lower bound on the size of the active set $A_l$. To this end, we introduce the following notation:
$$
|A_l| = (|A_l(i)|)_{i\in[k(n)]^d}
$$
and 
$$
||A_l|| = \sum_{i \in [k(n)]^d} |A_l(i)|.
$$

In the upcoming proposition we prove that we can use a multitype branching process as pointwise stochastic lower bound for $|A_l|$. For real-valued random variables $X$ and $Y$ we say that $X$ is a lower bound for $Y$ if for all $a \in \RR$ we have $\PP(X \geq a) \leq \PP(Y \geq a)$. Furthermore, we call a sequence $r = (r_k)_{k \geq 0}$ a \emph{distribution} if $\sum_{k=0}^\infty r_k = 1$ and $r_k \geq 0$ for all $k$. For every such sequence, we let $W_r:(0,1) \to \NN$ be a non-decreasing function satisfying
$$
|\{\omega|W_r(\omega) = k\}| = r_k.
$$
It follows that if we remove mass $\eta$ from the distribution $r$ and normalize, then this will be stochastically larger than $W_r^\eta = (W_r(\omega)|\omega < 1 - \eta)$. Indeed, the latter removes mass $\eta$, starting from the largest values of $W_r$. Using all this, we can state and prove the following proposition.

\begin{proposition}\label{prop:exploration_lower_bound_vector}
Let the assumptions of Theorem \ref{theorem:giant_component} be satisfied. Let $i \in [k(n)]^d$ and let $v \in C_i^n$ and denote by $(R_l,A_l,U_l)$ the exploration process started at $v$. Let $\eta \in (0,1)$ and assume that at most $\eta m(n)$ vertices of each compartment have already been exposed. Then for every $\delta > 0$ there exist a multitype branching process $(S_l)_l$, such that until $\delta m(n)$ vertices in at least one compartment have been exposed, we have that $|A_l|$ is stochastically bounded from below by $S_l$.

Moreover, the offspring distribution of $(S_l)_l$ can be chosen to be multinomial with parameters $N$ and $p = (p^i)$ where $p^i = \frac1{2d+1}\sum_{|j-i|_1\leq 1} E_j$. Furthermore, for $\delta$ small enough, $N$ can be chosen such that $\EE(N) > 1$.
\end{proposition}

\begin{proof}
We argue the existence by constructing a suitable multitype branching process. To this end, we first argue what happens when exploring a single vertex $w \in C_j^n$. Assume that at most $\eta m(n)$ vertices have been exposed in every compartment. Let $D^*$ be the size-biased degree distribution (see Remark \ref{remark:size_biased}) and set $Z_D = D^* - 1$. We define the distribution $q = (q_k)$ by $q_k = \PP(Z_D = k)$. Since at most a fraction $\eta$ of the vertices has been exposed, together with the fact that $d_n$ converges to $D$ as in Assumption \ref{assumption:convergent_deg_seq}, it follows that for $n$ large enough the amount of new neighbours found while exploring $w$ is bounded from below by $W_q^{2\eta}$. 

By symmetry, these $W_q^{2\eta}$ new vertices are equally likely to be in any of the neighbouring compartments of $w$, i.e., in the compartments $C_i^n$ with $|i-j|_1 \leq 1$. Therefore, we consider the random variables
$$
\bar{W}_{q,j}^{2\eta} \sim \Mult\left(W_q^{2\eta},\frac1{2d+1}\sum_{|i-j|_1 \leq 1} E_i\right), 
$$
which is a multinomial distribution. Here, $E_i(x) = \delta_{ix}$ for $x \in [k(n)]^d$.\\

Finally, we need to take into account that the new vertices may already have been exposed before. Because of the degree constraints, we remove these vertices from the active set. Note that $\bar W_{q,j}^{2\eta}(i)$ takes values in the set $\{0,1,\ldots,L\}$, where $L = W_q(1 - 2\eta)$. Therefore, if at most $\delta m(n)$ vertices have been exposed from any compartment $C_i^n$, then there are at most $L\delta m(n)$ possible half-edges connected to active vertices in $C_i^n$. On the other hand, there are at least $\mu_n(d_n,C_i^n) - L\delta m(n)$ half-edges left which are not connected to an active vertex. Therefore, the probability of choosing an active neighbour in that compartment is at most
$$
\frac{L\delta m(n)}{\mu_n(d_n,C_i^n) - L\delta m(n)}.
$$
From Assumption \ref{assumption:convergent_deg_seq} it follows that $\frac{\mu_n(d_n,C_i^n)}{m(n)}$ converges to $\frac12\EE(D)$ uniformly over the compartments. Therefore, given $\epsilon > 0$, for $n$ large enough the above is smaller than
$$
\gamma_\delta := \frac{L\delta}{\frac12\EE(D) - \epsilon - L\delta}
$$
provided $\delta > 0$ is small enough.

Collecting everything, we see that the number of new vertices found while exploring $w \in C_j^n$ is bounded from below by
$$
X_j \sim \bar W_{q,j}^{2\eta} - 2\sum_{|i-j|_1 \leq 1}\Bin(|\bar W_{q,j}^{2\eta}(i)|,\gamma_\delta)E_i.
$$
From this we can conclude that $X_j$ follows a multinomial distribution with parameters $N$ and $p^j = \frac1{2d+1}\sum_{|i-j|_1\leq 1} E_i$. In particular, for $N$ we have 
$$
N \sim W_q^{2\eta} - 2\Bin(W_q^{2\eta},\gamma_\delta).
$$

From this it follows that 
$$
\EE(N) = \EE(W_q^{2\eta})(1 - 2\gamma_\delta).
$$

Now note that $\lim_{\eta \to 0} \EE(W_q^{2\eta}) = \EE(W_q) = \EE(Z_D)$. Since by assumption $\EE(Z_D) > 1$ (see Remark \ref{remark:size_biased}), we find that for $\eta$ small enough we have $\EE(W_q^{2\eta}) > 1$. Furthermore, note that $\gamma_\delta$ tends to 0 as $\delta$ tends to 0. Combining the above, we find that we can choose $\delta$ and $\eta$ small enough so that $\EE(N) > 1$.
\end{proof}

Our next aim is to prove that if $||A_l||$ grows to size $\beta \log m(n)$, then it actually grows to size $m(n)^{\frac23}$ with high probability. For this, we will use the lower bound we found in Proposition \ref{prop:exploration_lower_bound_vector}. Before we can show this, we first need a lemma.

\begin{lemma}\label{lemma:bounding_stopping_time}
Let $X$ be a random variable such that $X \geq -1$, $\PP(X = -1) > 0$ and $\EE(X) > 0$. Define $S_n$ by
$$
S_n = S_0 + \sum_{i=1}^n Y_i,
$$
where $Y_i = \sum_{j=1}^{S_{i-1}} X_j$ with $X_1,\ldots X_{S_{i-1}}$ i.i.d. with distribution $X$. Suppose $S_0 = x > 0$ and define
$$
T(x) := \inf\{n||S_n| = 0\}. 
$$
Then there exists a $\lambda > 0$ such that
$$
\PP(T(x) < \infty) \leq e^{-\lambda x}.
$$
\end{lemma}
\begin{proof}
It suffices to prove the statement for $X$ bounded from above, since this only increases $T(x)$. Let $M(t) = \EE(e^{tX})$ be the moment generating function of $X$. Since $X$ is bounded from below, we have that $M(t)$ is defined for all $t \leq 0$. Note that $M(0) = 1$, $M'(0) = \EE(X) > 0$ and
$$
\lim_{t \to -\infty} M(t) \geq \lim_{t\to-\infty} \PP(X = -1) e^{-t} = \infty.
$$
From this, together with the continuity of $M(t)$, it follows that there exists a $\lambda > 0$ such that $M(-\lambda) = 1$. This implies that $Z_n = e^{-\lambda S_n}$ is a martingale. 

From the optional stopping theorem, we find that
$$
\EE\left(Z_{T(x)}\right) = \lim_{n\to\infty} \EE\left(Z_{n\wedge T(x)}\right) = \EE(Z_0) = e^{-\lambda x}.
$$
On the other hand, 
$$
\EE\left(Z_{T(x)}\right) \geq \PP(T < \infty),
$$
and hence we find that $\PP(T(x) < \infty) \leq e^{-\lambda x}$.
\end{proof}

Using this lemma, we can show that if the active set of the exploration process grows to size $\beta \log m(n)$, then the probability that the exploration process does not explore a large cluster is small. More precisely, we have the following proposition.

\begin{proposition}\label{prop:component_bound_hitting_time}
Let the assumptions in Proposition \ref{prop:exploration_lower_bound_vector} be satisfied. Suppose $\sum_{i=0}^l ||A_i|| \geq \beta \log m(n)$ for some $l$. Define
$$
T := \inf\{n|||A_n|| = 0\}.
$$
Then for $\beta$ large enough we have
$$
\PP(T < \infty) \leq 2m(n)^{-k}.
$$
\end{proposition}
\begin{proof}
Let $S_n$ be the lower bound for $|A_n|$ from Proposition \ref{prop:exploration_lower_bound_vector}. Then 
$$
|S_n| = \sum_{i \in [k(n)]^d} S_n(i) 
$$ 
is a lower bound for $||A_n||$, and in particular,
$$
\Sigma_l = \sum_{i=0}^l |S_i|
$$
is a lower bound for $\sum_{i=0}^l ||A_i||$.

Now assume that $\Sigma_l \geq \beta\log m(n)$ and define 
$$
\tilde T = \inf\{n||S_n| = 0\}.
$$
Then $\PP(T < \infty) \leq \PP(\tilde T < \infty)$. 

Note that we can write
$$
|S_{l+1}| = \sum_{i=1}^{|S_l|} \tilde X_i,
$$
where $\tilde X_1,\ldots,\tilde X_{|S_l|}$ are independent and distributed like $N$ as in Proposition \ref{prop:exploration_lower_bound_vector}. By telescoping, this implies that
$$
|S_{l+1}| = \sum_{j=1}^{\Sigma_l} X_j,
$$
where $X_1,\ldots,X_{\Sigma_l}$ are independent and equal in distribution to $N - 1$. 

Let $\lambda$ be as in Lemma \ref{lemma:bounding_stopping_time} for the random variable $X_1$. From Chernoff's bound it follows that
\begin{align*}
\PP\left(\sum_{j=1}^\sigma X_j \leq \frac{k}{\lambda}\log m(n)\right)
&\leq 
e^{\theta\frac{k}{\lambda}\log m(n)}\EE(e^{-\theta X_1})^\sigma
\\
&=
\exp\left(\theta\frac{k}{\lambda}\log m(n) + \sigma\log M(-\theta)\right)
\end{align*}
for all $\theta > 0$, where $M(t) = \EE(e^{t X_1})$. Since $M(0) = 1, M(-\lambda) = 1$ and $M'(0) = \EE(X_1) > 0$, there exists a $\lambda' \in (0,\lambda)$ such that $M(-\lambda') < 1$. In particular, this implies that $\log M(-\lambda') < 0$. From this it follows that for $\sigma \geq \beta\log m(n)$ we have 
$$
\PP\left(\sum_{j=1}^\sigma X_j \leq \frac{k}{\lambda}\log m(n)\right) \leq \exp\left(\left(\lambda'\frac{k}{\lambda} + \beta\log M(-\lambda')\right)\log m(n)\right).
$$
Using that $\log M(-\lambda') < 0$, we can take $\beta$ large enough such that
$$
\lambda'\frac{k}{\lambda} + \beta\log M(-\lambda') \leq -k
$$
so that
$$
\PP\left(\sum_{j=1}^\sigma X_j \leq \frac{k}{\lambda}\log m(n)\right) \leq m(n)^{-k}.
$$
From this we conclude that if $\Sigma_l \geq \beta \log m(n)$ for large enough $\beta$, we have that
$$
\PP\left(|S_{l+1}| < \frac{k}{\lambda}\log m(n)\right) \leq m(n)^{-k}.
$$

From this, we obtain that
\begin{align*}
\MoveEqLeft \PP(\tilde T < \infty)
\\
&\leq 
\PP\left(\left|S_{l+1}\right| \leq \frac{k}{\lambda}\log m(n)\right) + \PP\left(\tilde T < \infty \middle| \left|S_{l+1}\right| \geq 
\frac{k}{\lambda}\log m(n)\right) 
\\
&\leq 
2m(n)^{-k}. 
\end{align*}
Here, we applied Lemma \ref{lemma:bounding_stopping_time} to bound the second term. Since $\PP(T < \infty) \leq \PP(\tilde T < \infty)$, this proves the claim.
\end{proof}

We conclude this part by proving that if we repeatedly start the exploration process at a vertex in $C_i^n$ for some fixed $i \in [k(n)]^d$, then with high probability we find a component of size at least $m(n)^{\frac23}$ before $\delta m(n)$ vertices have been exposed. This follows from the fact that with high probability, each failed attempt uses at most $\beta\log m(n)$ vertices.  Carefully estimating this probability is necessary to deal with the diverging number of locally-large components that we find. This provides a major contrast with \cite{Dur07}, where showing that the result holds with high probability suffices.

\begin{proposition}\label{prop:existance_local_large_component}
The probability that the exploration process started (repeatedly) at a vertex in $C_i^n$ finds a component of size at least $m(n)^{\frac23}$ before a total of $\delta m(n)$ vertices have been exposed is at least
$$
\left(1 - 2m(n)^{-k}\right)^{\frac{\delta m(n)}{\beta \log m(n)}}\left(1 - (1-p)^{\frac{\delta m(n)}{\beta \log m(n)}}\right).
$$
Here, $p > 0$ is the probability that the branching process $(S_n)_n$ in Proposition \ref{prop:exploration_lower_bound_vector} survives indefinitely. 
\end{proposition}
\begin{proof}
Let $G$ denote the number of tries it takes before $S_n$ grows to size $m(n)^{\frac23}$. Then $G$ is geometrically distributed with parameter $\tilde p \geq p > 0$. Let $R_1,R_2,\ldots$ be a sequence of i.i.d. random variables representing the number of vertices exposed in a failed attempt. We need to prove that
$$
\PP\left(\sum_{i=1}^G R_i \leq \delta m(n)\right) \geq \left(1 - 2m(n)^{-k}\right)^{\frac{\delta m(n)}{\beta \log m(n)}}\left(1 - (1-p)^{\frac{\delta m(n)}{\beta \log m(n)}}\right).
$$

Now 
$$
\PP\left(\sum_{i=1}^G R_i \leq \delta m(n)\right) = \sum_{g=1}^\infty \PP\left(\sum_{i=1}^g R_i \leq \delta m(n)\right)\PP(G = g).
$$
For $g \leq \frac{\delta m(n)}{\beta \log m(n)}$ we have
$$
\PP\left(\sum_{i=1}^g R_i \leq \delta m(n)\right) \geq (1 - 2m(n)^{-k})^g \geq \left(1 - 2m(n)^{-k}\right)^{\frac{\delta m(n)}{\beta \log m(n)}},
$$
where the first inequality follows from Proposition \ref{prop:component_bound_hitting_time}. Using this, we find that
\begin{align*}
\PP\left(\sum_{i=1}^G R_i \leq \delta m(n)\right)
&\geq
\left(1 - 2m(n)^{-k}\right)^{\frac{\delta m(n)}{\beta \log m(n)}}\PP\left(G \leq \frac{\delta m(n)}{\beta \log m(n)}\right)
\\
&=
\left(1 - 2m(n)^{-k}\right)^{\frac{\delta m(n)}{\beta \log m(n)}}\left(1 - (1-\tilde p)^{\frac{\delta m(n)}{\beta \log m(n)}}\right).
\end{align*}
The desired bound now follow because $\tilde p \geq p$.
\end{proof}

\subsection{From local to global}\label{subsection:local_global}

In Proposition \ref{prop:existance_local_large_component} we have seen that if we start exploring from a vertex in $C_i^n$, with high probability we find a component of size at least $m(n)^{\frac23}$ at some point. In this section, we will first show that such a component spreads equally through all compartments near $C_i^n$. Since we could have started equally well from any other compartment, the idea is to show that with high probability, many of such locally-large components exist which together cover all compartments. It then remains to prove that these components are all connected with high probability, forming a large component which spreads through every compartment.

\subsubsection{Spreading through compartments}\label{subsubsection:spreading}

To see how the explored component spreads through neighbouring compartments, we use the multitype branching process found in Proposition \ref{prop:exploration_lower_bound_vector}. We will first show that, provided the branching process grows to a certain size, it actually does so exponentially fast with high probability. The following proposition is closely related to the large deviation results in \cite{Ath94} (see also \cite{AN72}). However, we need more precise information on the growth of the involved constants. 

\begin{proposition}\label{prop:exp_growth_GW}
Let $T = (T_l)_l$ be a Galton-Watson tree with bounded offspring distribution $N$ satisfying $\EE(N) > 1$ and $T_0 = 1$. Suppose there exists an $l$ such that $T_l \geq M_n$. Define
$$
p(n) = \inf\{p|T_p \geq M_n\} < \infty.
$$
Then for every $0 < \tau < 1$ and every $L > 1$ we have
$$
\PP\left(p(n) \leq LM_n^{\tau}\right) \geq 1 - 2e^{-2(L-1)M_n^{\tau}}
$$
for $n$ sufficiently large.
\end{proposition}
\begin{proof}
Because there exists and $l$ such that $T_l \geq M_n$, we know that every generation contains at least one vertex which survives until the tree grows to size $M_n$. We call such a vertex immortal. Every immortal vertex has at least one child that is also immortal. Moreover, since $E(N) > 1$, the probability of having only one immortal child is less than 1. Denote by $\overline N$ the offspring distribution $N$ conditioned to be at least 1. Then $\EE(\overline N) > 1$ and $\overline N$ is bounded since $N$ is bounded. 


Denote by $(Z_l)$ the Galton-Watson tree with offspring distribution $\overline N$. 
Since $\overline N \geq 1$, $Z_l$ is non-decreasing in $l$ and therefore we have
$$
p(n) \leq LM_n^{\tau} \quad \iff \quad Z_{L M_n^\tau} \geq M_n.
$$
This implies that
$$
\PP\left(p(n) \leq L M_n^\tau\right) = \PP\left(Z_{L M_n^\tau} \geq M_n\right).
$$

Since $\overline N \geq 1$, we have that $Z_n \geq n$. This implies that
$$
\PP(Z_{(\tilde L+1)n^\tau} \leq n) \leq \PP(Z_{\tilde Ln^\tau} \leq n | Z_0 = n^\tau).
$$

Now consider
$$
\PP(Z_{l+1} \leq (1+c)Z_l | Z_l = x)
$$
for some $c > 0$. We have
$$
Z_{l+1} \geq Z_l + \sum_{i=1}^x B_i,
$$
where the $B_i$ are independent Bernoulli random variables with parameter $p = \PP(\overline N) > 1 > 0$. Taking $c < p$, we find that
$$
\PP(Z_{l+1} \leq (1+c)Z_l | Z_l = x) \leq \PP\left(\frac{1}{x}\sum_{i=1}^x B_i \leq c\right).
$$
Using Hoeffding's inequality, we find
\begin{align*}
    \PP\left(\frac{1}{x}\sum_{i=1}^x B_i \leq c\right)
    &\leq
    \PP\left(\left|\sum_{i=1}^x B_i -  xp\right| > (p - c)x \right)
    \\
    &\leq
    2\exp\left(-2x(p-c)^2\right).
\end{align*}

In particular, taking $x = n^\tau$, we find that
$$
\PP(Z_{l+1} \leq (1+c)Z_l | Z_l = n^{\tau}) \leq 2\exp\left(-2n^\tau(p-c)^2\right) =: q_n.
$$

Now note that for $\alpha_n = \frac{1-\tau}{\log(1+c)}\log n$ we have $(1 + c)^{\alpha_n} n^\tau = n$. Let us denote by $\tilde B_i$ independent Bernoulli random variables with parameter $p_n = 1 - q_n$. Then
$$
\PP(Z_{\tilde Ln^\tau} \leq n | Z_0 = n^\tau) \leq \PP\left(\sum_{i=1}^{\tilde Ln^{\tau}} \tilde B_i \leq \alpha_n \right).
$$

In a similar fashion as above, we find that
$$
\PP\left(\sum_{i=1}^{\tilde Ln^{\tau}} \tilde B_i \leq \alpha_n \right) \leq 2\exp\left(-2\tilde Ln^{\tau}(\tilde Ln^\tau p_n - \alpha_n)^2\right)
$$
Since $\alpha_n \sim \log n$, we find that
$$
2\exp\left(-2\tilde Ln^{\tau}(\tilde Ln^\tau p_n - \alpha_n)^2\right) \leq 2\exp\left(-2\tilde Ln^{\tau}\right)
$$
for $n$ sufficiently large. The result now follows by combining all estimates above, inserting $M_n$ and noticing that $\tilde L = L - 1$.
\end{proof}

We are now ready  to estimate the sizes of the locally-large components we find while exploring the spread through the compartments. This result is the main reason why we have to resort to multitype branching processes. 

\begin{proposition}\label{prop:spreading_through_components}
Let the assumptions of Proposition \ref{prop:exploration_lower_bound_vector} be satisfied and denote by $C$ a component explored by the exploration process started at a vertex $v \in C_i^n$ for some $i \in [k(n)]^d$. Assume the active set of the exploration process reaches size $m(n)^{\frac23}$. Then for every $\epsilon > 0$ and $0 < \tau < \frac23$ we have 
$$
    \PP\left(|C \cap C_i^n| \geq \epsilon m(n)^{\frac23-\tau}\right)
    \geq
    \left(1 - 2\exp\left(-\frac18\epsilon^2m(n)^{\frac23-2\tau}\right)\right)\left(1 - 2 e^{-2m(n)^{\frac2d\tau}}\right)
$$
for sufficiently large  $n$. 
\end{proposition}
\begin{proof}
Let $(S_n)_n$ be the multitype branching process from Proposition \ref{prop:exploration_lower_bound_vector}. Assume there exists an $l$ such that $|S_l| \geq m(n)^{\frac23}$. Set 
$$
p(n) := \inf\left\{p \middle ||S_p| \geq m(n)^{\frac23}\right\} < \infty. 
$$
Then $|C \cap C_i^n| \geq S_{p(n)}(i)$. Therefore, it suffices to find a lower bound for
$$
\PP\left(S_{p(n)}(i) \geq \epsilon m(n)^{\frac23-\tau}\right).
$$

Note that $p(n)$ is a random variable. We have
$$
  \PP\left(S_{p(n)}(i) \geq \epsilon m(n)^{\frac23-\tau}\right) 
  = \sum_p \PP\left(S_p(i) \geq \epsilon m(n)^{\frac23-\tau}\right)\PP(p(n) = p).
$$

By Proposition \ref{prop:exp_growth_GW} we have (taking $L = 2$)
$$
\PP(p(n) \leq 2m(n)^{\frac2d\tau}) \geq 1 - 2 e^{-2m(n)^{\frac2d\tau}}.
$$

Now, for $p \leq 2 m(n)^{\frac2d\tau}$ we have
\begin{align*}
\PP\left(S_p(i) \geq \epsilon m(n)^{\frac23-\tau}\right)
&\geq
\PP\left(S_p(i) \geq \epsilon 2^{\frac d2} \frac{m(n)^{\frac23}}{p^{\frac d2}}\right)
\\
&\geq
1 - 2\exp\left(-\frac{\epsilon^2 2^d m(n)^{\frac23}}{8p^d}\right)
\\
&\geq
1 - 2\exp\left(-\frac18\epsilon^2m(n)^{\frac23-2\tau}\right).
\end{align*}
Here, the third line follows from Proposition \ref{prop:multitype_concentration} (with $r = 0$). 

If we now collect everything, we find that
\begin{align*}
    \MoveEqLeft \PP\left(S_{p(n)}(i) \geq \epsilon m(n)^{\frac23-\tau}\right)
    \\
    &\geq
    \left(1 - 2\exp\left(-\frac18\epsilon^2m(n)^{\frac23-2\tau}\right)\right) \PP(p(n) \leq 2m(n)^{\frac2d\tau})
    \\
    &\geq
    \left(1 - 2\exp\left(-\frac18\epsilon^2m(n)^{\frac23-2\tau}\right)\right)\left(1 - 2 e^{-2m(n)^{\frac2d\tau}}\right).
\end{align*}
This concludes the proof.
\end{proof}

\subsubsection{Connecting local components}

Proposition \ref{prop:existance_local_large_component} and \ref{prop:spreading_through_components} together give a lower bound on the probability that there exists a component of at least size $m(n)^{\frac23}$ of which at least $\epsilon m(n)^{\frac23-\tau}$ vertices are in a given compartment $C_i^n$.\\

For $i \in [k(n)]^d$, denote by $I_i$ the indicator random variable of the event that there exists a component as in Section \ref{subsubsection:spreading}, where the exploration is started in compartment $C_i^n$. It follows from Proposition \ref{prop:existance_local_large_component} and \ref{prop:spreading_through_components} that $\PP(I_i = 1) = a_nb_n$, where
$$
a_n = \left(1 - 2m(n)^{-k}\right)^{\frac{\delta m(n)}{\beta \log m(n)}}\left(1 - (1-p)^{\frac{\delta m(n)}{\beta \log m(n)}}\right)
$$
and
$$
b_n = \left(1 - 2\exp\left(-\frac18\epsilon^2m(n)^{\frac23-2\tau}\right)\right)\left(1 - 2 e^{-2m(n)^{\frac2d\tau}}\right).
$$

However, the random variables $I_i$ are not independent. Nonetheless, we have
\begin{align*}
\PP(I_i = 1 \mbox{ for all } i =\in [k(n)]^d) 
&=
1 - \PP(I_i = 0 \mbox{ for some } i \in [k(n)]^d)
\\
&\geq
1 - \sum_{i\in [k(n)]^d} \PP(I_i = 0)
\\
&= 
1 - k(n)^d\PP(I_1 = 0)
\\
&=
1 - k(n)^d(1 - a_nb_n).
\end{align*}

Next, we want to show that sufficiently large components from neighbouring compartments are actually connected with high probability. For this, we need the following lemma.
\begin{lemma}\label{lemma:connection_disjoint_sets}
Let $i,j \in [k(n)]^d$ be such that $|i-j|_1 \leq 1$. Let $A \subset C_i^n$ and $B \subset C_j^n$. Assume that the vertices in $A$ and $B$ have degree at least 1. Then there exists a $C > 0$ (depending on $d$) such that for $n$ large enough the probability that there is no edge between $A$ and $B$ in $G_n$ is at most 
$$
\exp\left(-\frac{C|A||B|}{m(n)}\right).
$$
\end{lemma}
\begin{proof}
Since all vertices have degree at least 1, the total degree in $A$ and $B$ is at least $|A|$ respectively $|B|$. On the other hand, because of the convergence of the degree sequence of $G_n$, we know that for $n$ large enough the total degree in each compartment is at most $2\EE(D)m(n)$. This implies that a half-edge at a vertex can be connected to at most $(4d+2)\EE(D)m(n)$ half-edges. With these observations, the result follows from a similar reasoning as in \cite[Lemma 20]{BR15}.
\end{proof}

From Lemma \ref{lemma:connection_disjoint_sets} it follows that the probability that two neighbouring components are connected is more than
$$
c_n := 1 - \exp\left(-C\epsilon^2m(n)^{\frac13-2\tau}\right).
$$

Collecting everything, we find that there exists a component $C$ with $|C \cap C_i^n| \geq \epsilon m(n)^{\frac23-\tau}$ for all $i \in [k(n)]^d$ with probability at least
$$
(1 - k(n)^d(1 - a_nb_n))(1 - dk(n)c_n).
$$
Here, the $d$ in the second factor comes from the observation that every compartment has $2d$ neighbouring compartment. Since every neighbour relation is counted twice when summing over all compartment, we have to divide by two.\\

From the above discussion, we obtain the following.

\begin{proposition}\label{prop:large_global_component}
For every $\tau \in (0,\frac16)$ we have that with high probability there exists a component $C$ in $G_n$ such that for all $i \in [k(n)]^d$ we have
$$
|C \cap C_i^n| \geq \epsilon m(n)^{\frac23-\tau}.
$$
\end{proposition}
\begin{proof}
Following the reasoning above, it remains to show that
$$
\lim_{n\to\infty} (1 - k(n)^d(1 - a_nb_n))(1 - dk(n)c_n) = 1.
$$

In order to do this, we observe that is suffices to prove that
\begin{equation}\label{eq:convergence_condition}
\lim_{n\to\infty} a_n^{k(n)^d} = \lim_{n\to\infty} b_n^{k(n)^d} = \lim_{n\to\infty} (1 - c_n)^{dk(n)} = 1
\end{equation}

Indeed, suppose $0 < x_n < 1$ and assume $\lim_{n\to\infty} x_n^{f(n)} = 1$. Then we have $\lim_{n\to\infty} f(n)\log(x_n) = 0$. But $\log(x_n) \leq x_n - 1 \leq 0$, and hence, by the squeeze theorem we find that $\lim_{n\to\infty} f(n)(x_n - 1) = 0$ from which it follows that $\lim_{n\to\infty} 1 - f(n)(1 - x_n) = 1$.\\

Let us prove that \eqref{eq:convergence_condition} holds. We will only show this for $a_n$, the result for $b_n$ and $c_n$ being proven similarly (the conditions on $\tau$ being needed there to have the desired decay).  

Since by assumption $\lim_{n\to\infty} \frac{k(n)^dm(n)}{n} = 1$, we have for $n$ large that $k(n)^d \approx \frac{n}{m(n)}$. Therefore, we have asymptotically
$$
a_n^{k(n)^d} \approx \left(1 - 2m(n)^{-k}\right)^{\frac{\delta n}{\beta \log m(n)}}\left(1 - (1-p)^{\frac{\delta m(n)}{\beta \log m(n)}}\right)^{\frac{n}{m(n)}}.
$$

For the first factor, taking logarithms, we have
$$
\frac{\delta n}{\beta \log m(n)}\log\left(1 - 2m(n)^{-k}\right) \approx -\frac{\delta n}{\beta m(n)^k \log m(n)}
$$
where we used that $\log(1 - x) \approx -x$. Since by assumption $k$ is such that $\lim_{n\to\infty} nm(n)^{-k} = 0$, the above converges to 0 and therefore
$$
\lim_{n\to\infty}  \left(1 - 2m(n)^{-k}\right)^{\frac{\delta n}{\beta \log m(n)}} = 1.
$$

In a similar way, the second factor converges to 1 if 
$$
\lim_{n\to\infty} \frac{n}{m(n)}(1-p)^{\frac{\delta m(n)}{\beta \log m(n)}} = 0.
$$
This again follows from the assumptions that $\lim_{n\to\infty} nm(n)^{-k} = 0$, since the second factor decays exponentially in $m(n)$. This concludes the proof.
\end{proof}

\subsection{The size of the giant component} \label{subsection:size_giant}

So far, we have shown that with high probability there exists a large component spreading through all compartments. It remains to show that there is only one such component, and that its size is asymptotically $(1 - \rho) n$, where $\rho$ is the extinction probability of the Galton-Watson tree as explained in Remark \ref{remark:size_biased}.\\

From Proposition \ref{prop:component_bound_hitting_time} we obtain the following identification of the largest component in the compartment model $G_n$. Note that this also proves the final statement of Theorem \ref{theorem:giant_component}.

\begin{proposition}\label{prop:giant_component_discription}
Let the assumptions of Theorem \ref{theorem:giant_component} be satisfied. Then with high probability the largest component in $G_n$ is equal to
$$
\{x||C_x| \geq \beta\log m(n)\},
$$
where $C_x$ denotes the component of $G_n$ containing $x$.
\end{proposition}
\begin{proof}
By Proposition \ref{prop:large_global_component} we know that $C \subset \{x||C_x| \geq \beta\log m(n)\}$ with high probability. The claim now follows once we show that 
$$
\PP\left(\{x||C_x| \geq \beta\log m(n)\} \subset C\right)
$$
goes to 1. For this, it suffices to prove that
$$
\PP(|C_x| \geq \beta\log m(n) \mbox{ and } x \notin C \mbox{ for some } x)
$$
goes to 0. Note that there are at most $\frac{n}{\beta \log m(n)}$ components of size larger than $\beta\log m(n)$. Therefore, the above probability is bounded above by
$$
\frac{n}{\beta\log m(n)}\PP\left(|\tilde C| \geq \beta\log m(n) \mbox{ and } \tilde C \cap C = \emptyset\right).
$$
By conditioning we have
\begin{align*}
\MoveEqLeft \PP\left(|\tilde C| \geq \beta\log m(n) \mbox{ and } \tilde C \cap C = \emptyset\right)
\\
&\leq 
\PP\left(|\tilde C| < m(n)^{\frac23}\middle||\tilde C| \geq \beta\log m(n)\right) + \PP\left(\tilde C \cap C = \emptyset\middle||\tilde C| \geq m(n)^{\frac23}\right).
\end{align*}

From Proposition \ref{prop:component_bound_hitting_time} it follows that
$$
\PP\left(|\tilde C| < m(n)^{\frac23}\middle||\tilde C| \geq \beta\log m(n)\right) \leq m(n)^{-k}.
$$
Furthermore, an argument similar to the proof of Lemma \ref{lemma:connection_disjoint_sets} gives us that
$$
\PP\left(\tilde C \cap C = \emptyset\middle||\tilde C| \geq m(n)^{\frac23}\right) \leq \exp\left(-Cm(n)^{\frac13-\tau}\right).
$$
Here we used that the component $C$ contains at least $\epsilon m(n)^{\frac23-\tau}$ vertices from each compartment. 

Since by assumption $\lim_{n\to\infty} nm(n)^{-k} = 0$, it follows that
$$
\lim_{n\to\infty} \frac{n}{\beta\log m(n)}\left(m(n)^{-k} + \exp\left(-Cm(n)^{\frac13-\tau}\right) \right) = 0
$$
as long as we take $\tau < \frac13$. This completes the proof.
\end{proof}

From Proposition \ref{prop:giant_component_discription} it follows that we are done once we show that
$$
\frac1n\left|\{x||C_x| \geq \beta\log m(n)\}\right| \to 1 - \rho
$$
in probability, where $\rho$ is as in Remark \ref{remark:size_biased}. For this, we first prove the following result.

\begin{proposition}
Let the assumptions of Theorem \ref{theorem:giant_component} be satisfied. Let $\rho$ be as in Remark \ref{remark:size_biased} and let $C$ be a component of $G_n$. Then
\begin{equation}\label{eq:prob_convergence_extinction}
\lim_{n\to\infty} \PP(|C| \leq \beta\log m(n)) = \rho.
\end{equation}
\end{proposition}
\begin{proof}
Define $Z_t$ by
$$
Z_t = Z_0 + \sum_{i=1}^t X_i
$$
where $Z_0$ has distribution $D$ and $X_1,\ldots,X_t$ are i.i.d. with distribution $Z_D$ as in Remark \ref{remark:size_biased}. Furthermore, define $S_t = \sum_{i=0}^t Z_i$. Suppose we explore the component $C$ using the exploration process $(R_t,A_t,U_t)$, where we do not track to which compartment the vertices belong. Since during the exploration we might have collisions, we find that $|A_t| \leq Z_t$ and $|R_t| \leq S_t$. Define
$$
\tau := \inf\{t|A_t = \emptyset\}
$$
and 
$$
T := \inf\{t|Z_t = 0\}.
$$

Since $|A_t| \leq Z_t$ we have $\tau \leq T$. Furthermore, note that $\PP(|C| \leq \beta\log m(n)) = \PP(|R_\tau| \leq \beta\log m(n))$. Because $|R_t| \leq S_t$, it follows that
$$
\PP(|R_\tau| \leq \beta\log m(n)) \geq \PP(S_T \leq \beta\log m(n)).
$$
It holds that
$$
\lim_{n\to\infty} \PP(S_T \leq \beta\log m(n)) = \PP(T < \infty) = \rho.
$$
Indeed, since $S_t$ increases by at least 1 in every step, we have
$$
\PP(S_T \leq \beta\log m(n)) \leq \PP(T \leq \beta\log m(n)) \to \rho.
$$
On the other hand, note that
\begin{align*}
\PP(S_T \leq \beta\log m(n)) 
&\geq 
\PP(S_T \leq \beta\log m(n)|T \leq \tilde\beta\log m(n))\PP(T \leq \tilde\beta\log m(n))
\\
&= 
(1 - o(1))\PP(T \leq \tilde\beta\log m(n))
\\
&\to \rho
\end{align*}
provided $\tilde\beta$ is small enough. This can for example be proven using Chernoff's bound.

We conclude that
\begin{equation}\label{eq:liminf}
\liminf_{n\to\infty} \PP(|R_\tau| \leq \beta\log m(n)) \geq \rho.
\end{equation}

For the reverse inequality, we write
\begin{align*}
\MoveEqLeft \PP(S_T > \beta\log m(n))
\\
&= \PP(S_T > \beta\log m(n), S_t = |R_t| \mbox{ for all } t \leq \beta\log m(n))
\\
&\qquad + \PP(S_T > \beta\log m(n), S_t > |R_t| \mbox{ for some } t \leq \beta\log m(n)).
\end{align*}

Note that
$$
\PP(S_T > \beta\log m(n), S_t = |R_t| \mbox{ for all } t \leq \beta\log m(n)) \leq \PP(|R_{\tau}| > \beta\log m(n)).
$$

We will show that the second term vanishes. For this, observe that it suffices to prove that the probability of a collision before $\beta\log m(n)$ vertices are exposed vanishes. For this, observe that if we explore half-edge $t \leq \beta\log m(n)$, then there are at most $\beta \log m(n)$ vertices it can attach to to form a collision. Since the degree sequence converges, there exists a constant $C > 0$ such that with high probability there are at most $C\log m(n)$ half-edges that lead to a collision. For the same reason, there are at least $c m(n)$ half-edges to choose from in total. Therefore, the probability of causing a collision when exploring half-edge $t$ is at most $\tilde C\frac{\log m(n)}{m(n)}$. It follows that the probability of a collision in the first $\beta\log m(n)$ exploration steps is of order $\frac{(\log m(n))^2}{m(n)}$, which tends to 0. 

Collecting everything, we find that
$$
\PP(S_T > \beta\log m(n)) \leq \PP(|R_{\tau}| > \beta\log m(n)) + o(1),
$$
from which it follows that
$$
\PP(|R_{\tau}| \leq \beta\log m(n)) \leq \PP(S_T \leq \beta\log m(n)) + o(1).
$$

From this we conclude that
$$
\limsup_{n\to\infty} \PP(|R_{\tau}| \leq \beta\log m(n)) \leq \lim_{n\to\infty} \PP(S_T \leq \beta\log m(n)) = \rho.
$$

Together with the inequality in \eqref{eq:liminf} this completes the proof.
\end{proof}

\subsection{Proof of Theorem \ref{theorem:giant_component}}

With all preparations done, we are finally ready to prove Theorem \ref{theorem:giant_component}.

\begin{proof}[Proof of Theorem \ref{theorem:giant_component}]

From Proposition \ref{prop:giant_component_discription} it follows that we are done once we show that
$$
\lim_{n\to\infty} \frac1n|\{x||C_x| \geq \beta\log m(n)\}| = 1 - \rho
$$
in probability.

To this end, define for $x \in V_n$ the random variables $Y_x^n$, where $Y_x^n = 1$ if $|C_x| \leq \beta \log m(n)$ and 0 otherwise. Then
$$
|\{x||C_x| \leq \beta\log m(n)\}| = \sum_{x = 1}^{k(n)^dm(n)} Y_x.
$$
Note that by \eqref{eq:prob_convergence_extinction} we have
$$
\lim_{n\to\infty} \EE(Y_x^n) = \lim_{n\to\infty} \PP(|C_x| \leq \beta\log m(n)) = \rho.
$$
Therefore, we find that
\begin{multline}
\lim_{n\to\infty} \PP\left(\left|\frac1n|\{x||C_x| \geq \beta\log m(n)\}| - \rho\right| > \epsilon\right) \\
= \lim_{n\to \infty} \PP\left(\left|\sum_{x = 1}^{k(n)^dm(n)} Y_x - k(n)^dm(n)\rho\right| \geq n\epsilon\right).
\end{multline}

By Chebyshev's inequality, we have
$$
\PP\left(\left|\sum_{x = 1}^{k(n)^dm(n)} Y_x - k(n)^dm(n)\rho\right| \geq n\epsilon\right) \leq \frac{\Var\left(\sum_{x=1}^{k(n)^dm(n)} Y_x\right)}{\epsilon^2n^2}.
$$

Note that 
\begin{multline}
\Var\left(\sum_{x=1}^{k(n)^dm(n)} Y_x\right)
\\
\leq k(n)^dm(n) + \sum_{x\neq y} \Cov(Y_x,Y_y) \leq k(n)^dm(n) + k(n)^{2d}m(n)^2\Cov(Y_1,Y_2).
\end{multline}
We can compute
$$
\Cov(Y_1,Y_2) = \PP(Y_1 = 1,Y_2 = 1) - \PP(Y_1 = 1)\PP(Y_2 = 1).
$$
To estimate this, we consider two independent exploration processes starting at vertex 1 and 2 where we couple them once they meet. Following a reasoning similar to \cite[Lemma 2.3.4]{Dur07}, we find that 
$$
\Cov(Y_1,Y_2) \leq \frac{C(\beta \log m(n))^2}{m(n)}
$$
for some $C > 0$.

Altogether, we obtain
$$
\Var\left(\sum_{x=1}^{k(n)m(n)} Y_x\right) \leq Ck(n)^{2d}m(n)(\log m(n))^2
$$
for some (possibly different) constant $C > 0$. Plugging this into the equation above and using that $\lim_{n\to\infty} \frac{k(n)^dm(n)}{n} = 1$ and $\lim_{n\to\infty} m(n) = \infty$, we find that
$$
\lim_{n\to\infty} \PP\left(\left|\sum_{x = 1}^{k(n)^dm(n)} Y_x - k(n)^dm(n)\rho\right| \geq n\epsilon\right) = 0.
$$
Putting everything together, we obtain
$$
\frac1n|\{x||C_x| \leq \beta\log m(n)\}| \to \rho
$$
in probability, which implies that
$$
\frac1n|\{x||C_x| \geq \beta\log m(n)\}| \to 1 - \rho
$$
in probability as desired.
\end{proof}

\section{Difference with standard configuration model} \label{section:counterexample}

We conclude by considering an example to see the difference between the compartment model on torus and the standard configuration model. This example also shows that some condition on the number of vertices $m(n)$ per compartment is necessary when we want to keep the conditions on the degree sequence in line with the standard configuration model.\\

For our example, let $D$ be random variable taking values in the non-negative integers. Assume that $\PP(D \leq 1) = p > 0$ and $\EE(D(D-2)) > 0$. Let $d_n$ be a degree sequence on $n$ vertices converging to $D$ in the sense of Assumption \ref{assumption:convergent_deg_seq} (without the compartments). Let $G(d_n)$ be the random graph obtain from the standard configuration model on $n$ vertices with degree sequence $d_n$. Then (see e.g. \cite{BR15})
$$
\lim_{n\to\infty} \frac{L_1(G(d_n))}{n} = 1 - \rho
$$
in probability, where $\rho$ is the extinction probability of the Galton-Watson tree associated to $D$ as in Remark \ref{remark:size_biased}. In particular, because $\EE(D(D-2)) > 0$ it holds that $\rho < 1$. We thus see that with high probability the graph $G(d_n)$ contains a giant component.\\

We will now prove that under the same conditions, the compartment model on the torus does not contain a giant component with high probability if we assume the compartment contain a fixed number of vertices. This is caused only by the assumption that $\PP(D \leq 1) > 0$. The result below considers the circle ($d=1$). Afterwards, we will remark how this may be extended to higher dimensions.

\begin{proposition}\label{prop:counter_example}
Let $D$ be a random variable taking values in the non-negative integers such that $\PP(D \leq 1) = p > 0$. Let $d_n$ be a degree sequence sampled independently and uniformly from $D$. Let $G_n$ be the compartment model on the circle (i.e. $d = 1$) with degree sequence $d_n$ and assume that $m(n) = \lambda \geq 1$ for all $n$. Then 
$$
\lim_{n\to\infty} \frac{L_1(G_n)}{n} = 0
$$
in probability.
\end{proposition}
\begin{proof}
Observe that if all vertices in a compartment have degree 0 or 1, then no component can cross this compartment. As a consequence, the size of components is bounded by the maximum distance between such compartments multiplied by $\lambda$. 

Note that with probability $p^\lambda$ a compartment contains only degree 0 or 1 vertices. Since for different compartments these events are independent, the distance between such compartments is geometrically distributed with parameter $p^\lambda$. Moreover, since $\lambda \geq 1$, we have at most $n$ such intervals.

Let $X_1,\ldots,X_n$ be independent random variables with a geometric distribution with parameter $p^\lambda$. By the above, it follows that the size of the largest component is bounded by $\max_{i=1}^n X_i$. 

Now let $\epsilon > 0$. Then 
$$
\PP\left(\max_{i=1}^n X_i \leq \epsilon n\right) = \prod_{i=1}^n \PP(X_i \leq \epsilon n) = \left(1 - (1 - p^\lambda)^{\epsilon n}\right)^n
$$

Now,
$$
\lim_{n\to\infty} \left(1 - (1 - p^\lambda)^{\epsilon n}\right)^n = 1.
$$
To see this, note that
$$
\log\left(\left(1 - (1 - p^\lambda)^{\epsilon n}\right)^n\right) = n\log\left(1 - (1 - p^\lambda)^{\epsilon n}\right) \approx n(1 - p^\lambda)^{\epsilon n},
$$
which goes to 0 since $0 < (1 - p^\lambda)^\epsilon < 1$.

Using the above, we find that
$$
\lim_{n\to\infty} \PP\left(\max_{i=1}^n X_i \leq \epsilon n\right) = 1
$$
Collecting everything, it follows that
\begin{align*}
    \lim_{n\to\infty} \PP\left(\frac{L_1(G_n)}{n} > \epsilon \right)
    &\leq
   \lim_{n\to\infty} \PP\left(\max_{i=1}^n X_i > n\epsilon \lambda^{-1} \right)
   \\
   &=
   1 - \lim_{n\to\infty} \PP\left(\max_{i=1}^n X_i \leq n\epsilon\lambda^{-1} \right)
   \\
   &=
   0.
\end{align*}
We conclude that
$$
\lim_{n\to\infty} \frac{L_1(G_n)}{n} = 0
$$
in probability.
\end{proof}

The result of Proposition \ref{prop:counter_example} remains true in higher dimensions, at least under the additional assumption that $p^\lambda$ (i.e., the probability of a compartment with only vertices of degree at most 1) is sufficiently large. The reasoning makes use of the phase transition in site-percolation on the lattice $(\ZZ/k(n)\ZZ)^d$. Indeed, the vertices in $(\ZZ/k(n)\ZZ)^d$ represent compartments. Since components in the compartment model cannot cross compartments with only degree 1 vertices, they are  restricted to compartments that form components in $(\ZZ/k(n)\ZZ)^d$ after removing sites with probability $p^\lambda$. If this is sufficiently large, then only components of size $\log(k(n))$ remain with high probability. In that case, the components in the compartment model have at most size $m(n)\log(k(n))$, which is $o(n)$. Hence, if $p^\lambda$ is large enough, the associated compartment model on the torus does not have a giant component.

\begin{remark}\label{remark:compartment_size_percolation}
The result in Proposition \ref{prop:counter_example} can actually be extended to slightly larger compartment sizes (as long as $d = 1$). Indeed, the same reasoning also works when $m(n) = \lambda \log n$ as long as $\lambda < -\frac{1}{\log p}$. This proves that at least for some degree distributions $D$ (and for $d = 1$) it is actually necessary for $m(n)$ to tend to infinity in order to see a giant component. This also underpins the idea that there is an interplay between the assumptions on the degree sequence and compartment for the emergence of a giant component. 
\end{remark}

Remark \ref{remark:compartment_size_percolation} suggest that there is a phase transition in the behaviour of the giant component in the compartment model on a circle depending on the size of the compartments. In particular, we have the following conjecture.

\begin{conjecture}\label{conjecture:phase_transition}
Let $D$ be a random variable taking values in the non-negative integers such that $\PP(D \leq 1) = p > 0$. Let $d_n$ be a degree sequence sampled independently and uniformly from $D$. Let $G_n$ be the compartment model on the circle (i.e. $d = 1$) with degree sequence $d_n$ and assume that $m(n) = \lambda\log n$ for all $n$. Then there exists a critical value $\lambda^*$ at which a phase transition occurs in the existence of a giant component in graph $G_n$.
\end{conjecture}

It should be noted that Conjecture \ref{conjecture:phase_transition} will likely not hold for general degree distributions, and relies on the assumption that $\PP(D \leq 1) > 0$. As mentioned in Remark \ref{remark:compartment_size_percolation}, we expect that there is intricate interplay between assumptions on the degree sequence and compartment size in order to see a giant component. Therefore, any potential critical size of $m(n)$ is likely to rely on properties of the degree distribution $D$.  

\begin{remark}[Percolation]
As a consequence of the results in this section, we find that percolation for the compartment model on the circle with fixed size compartments looks very unusual. Indeed, if we independently keep edges with probability $p < 1$, then the probability that a vertex in the resulting graph has degree at most 1 is greater than 0. The argument above then shows that this graph does not have a giant component. Therefore,  the percolation threshold is $p^* = 1$.  
\end{remark}

\appendix

\section{Appendix: Concentration inequalities} \label{section:concentration}

In this appendix we obtain the vector-valued extension of the classical result on concentration inequalities by McDiarmid. This is a special case of the results in \cite{KXBS21}. Since we do not need such generality, we state McDiarmid's theorem (\cite{McD89}) for completeness and derive the vector-valued extension from this.

\begin{theorem}\label{theorem:McDiarmid}
Let $f:\RR^n \to \RR$ be a function and $X_1,\ldots,X_n$ independent real-valued random variables. Let $c_1,\ldots,c_n$ be constants such that
$$
\sup_{x_1,\ldots,x_i,x_i',\ldots,x_n} f(x_1,\ldots,x_i,\ldots,x_n) - f(x_1,\ldots,x_i',\ldots,x_n) \leq c_i.
$$
Then for every $\epsilon > 0$ we have
$$
\PP(|f(X_1,\ldots,X_n) - \EE(f(X_1,\ldots,X_n))| \geq \epsilon) \leq 2\exp\left(-\frac{\epsilon^2}{2\sum_{i=1}^n c_i^2}\right).
$$
\end{theorem}

We will prove a similar estimate when $F$ is vector-valued. For $x \in \RR^d$ we denote by $||x||_\infty$ the sup-norm of $x$, i.e.,
$$
||x||_\infty = \max_{i=1,\ldots,n} |x_i|.
$$
We obtain the following extension of McDiarmid's theorem.

\begin{theorem}\label{theorem:McDiarmid_vector}
Let $F:\RR^n \to \RR^d$ be a function and $X_1,\ldots,X_n$ independent, real-valued random variables. Let $c_1,\ldots,c_n$ be constants such that
$$
\sup_{x_1,\ldots,x_i,x_i',\ldots,x_n} ||F(x_1,\ldots,x_i,\ldots,x_n) - F(x_1,\ldots,x_i',\ldots,x_n)||_\infty \leq c_i.
$$
Then for every $\epsilon > 0$ we have
$$
\PP(||F(X_1,\ldots,X_n) - \EE(F(X_1,\ldots,X_n))||_\infty \geq \epsilon) \leq 2d\exp\left(-\frac{\epsilon^2}{2\sum_{i=1}^n c_i^2}\right).
$$
\end{theorem}
\begin{proof}
For every $j = 1,\ldots,d$ we can apply Theorem \ref{theorem:McDiarmid} to $F_j:\RR^n \to \RR$, the $j$-th component of $F$. This gives us that
$$
\PP(|F_j(X_1,\ldots,X_n) - \EE(F_j(X_1,\ldots,X_n))| \geq \epsilon) \leq 2\exp\left(-\frac{\epsilon^2}{2\sum_{i=1}^n c_i^2}\right).
$$
We can now estimate
\begin{align*}
    \MoveEqLeft
    \PP(||F(X_1,\ldots,X_n) - \EE(F(X_1,\ldots,X_n))||_\infty \geq \epsilon)
    \\
    &=
    \PP\left(|F_j(X_1,\ldots,X_n) - \EE(F_j(X_1,\ldots,X_n))| \geq \epsilon \mbox{ for some } j = 1,\ldots,d\right)
    \\
    &\leq
    \sum_{j=1}^d \PP(|F_j(X_1,\ldots,X_n) - \EE(F_j(X_1,\ldots,X_n))| \geq \epsilon)
    \\
    &\leq
    2d\exp\left(-\frac{\epsilon^2}{2\sum_{i=1}^n c_i^2}\right),
\end{align*}
which completes the proof.
\end{proof}



\smallskip

\textbf{Acknowledgement}
This research was sponsored by the Army Research Office and was accomplished under Cooperative Agreement Number W911NF-20-2-0187. 

\printbibliography

\end{document}